\documentclass[11pt,reqno]{amsart}
\usepackage{fancyhdr}
\usepackage[british]{babel}

\usepackage{amssymb, amsmath, amsthm, amsfonts, enumerate}
\usepackage{amsmath,setspace,scalefnt}
\usepackage[dvipsnames,svgnames,table]{xcolor}
\usepackage{graphicx,tikz,caption,subcaption}
\usetikzlibrary{patterns}
\usetikzlibrary{shapes}
\usepackage{fullpage}
\usepackage{hyperref}
\usepackage{multicol}
\usepackage{float}
\usepackage{cleveref}
\usepackage[normalem]{ulem}
\usepackage{enumerate}
\usepackage{comment}

\usepackage[margin=2.5cm]{geometry}

\definecolor{gray}{rgb}{0.25, 0.25, 0.25}

\newtheorem{theorem}{Theorem}[section]
\newtheorem{lemma}[theorem]{Lemma}

\newtheorem*{theorem*}{Theorem}

\theoremstyle{definition}

\newcounter{propcounter}

\theoremstyle{plain}

\theoremstyle{definition}

\theoremstyle{definition}

\theoremstyle{definition}

\theoremstyle{definition}

\theoremstyle{definition}

\theoremstyle{definition}

\def\cf{\mathcal F}

\title{Note on the codegree version of the Erd\H{o}s--Ko--Rado theorem}
\author{Luyining Gan}
\address{School of Mathematical Sciences, Beijing University of Posts and Telecommunications, Beijing, China; Key Laboratory of Mathematics and Information Networks (BUPT), Ministry of Education, Beijing, China}
\email{elainegan@bupt.edu.cn}
\thanks{LG was partially supported by the National Natural Science Foundation of China (12401446) and the Fundamental Research Funds for the Central Universities}

\author{Jie Han}
\address{School of Mathematics and Statistics, Beijing Institute of Technology, Beijing, China}
\email{han.jie@bit.edu.cn}
\thanks{JH was partially supported by the National Natural Science Foundation of China (12371341).}

\author{Seonghyuk Im}
\address{Center for AI and Natural Sciences, Korea Institute for Advanced Study (KIAS), Seoul, South Korea}
\email{seonghyuk@kias.re.kr}
\thanks{SI was partially supported by the National Research Foundation of Korea (NRF) grant funded by the Korea government (MSIT) No. RS-2023-00210430, supported by the Institute for Basic Science (IBS-R029-C4), and supported by a KIAS individual
Grant (AP109501) at Korea Institute for Advanced Study.}

\begin{document}

\begin{abstract}
    Kupavskii proved a codegree version of the Erd\H{o}s--Ko--Rado theorem by showing that for an intersecting family $\mathcal{F} \subseteq \binom{[n]}{k}$ with $n \geq 2k +3d/(1-d/k)$, the minimum $d$-degree of $\mathcal{F}$ is at most $\binom{n-d-1}{k-d-1}$.
    Huang and Zhang improved the bound on $n$ to $n \geq 2k+2d-3$.
    In this short note, we prove that if $d = k-1$, then the bound on $n$ can be improved to $2k + \sqrt{2k} + O(1)$.
    In addition, we extend our method to show that the bound on $n$ can be improved to $2k + 7k^{2/3}+O(k^{1/3})$ when $d=k-2$.
\end{abstract}

\maketitle

\section{Introduction}

Given a positive integer $n$, denote by $[n]$ the set $\{1, 2, \dots, n\}$.
Let $\cf\subseteq \binom{[n]}{k}$ be a family of $k$-subsets of $[n]$. 
A \emph{vertex cover} of $\cf$ is a set $S\subseteq V(\cf)$ such that $S\cap e \neq \emptyset$ for all $e\in \cf$. 
The family $\cf$ is called \emph{intersecting} if $A\cap B\neq \emptyset$ for all $A, B\in \cf$.
The family $\mathcal{F}_v = \{ e \subseteq [n] \mid v \in e, |e|=k\} \subseteq \binom{[n]}{k}$ is called a \emph{complete star centered at $v$}, or simply a \emph{complete star}. 

The Erd\H{o}s--Ko--Rado (EKR) theorem~\cite{EKR61} is a classical result in extremal combinatorics, which states that for $n \ge 2k$, if a family $\cf\subseteq \binom{[n]}{k}$ is intersecting, then $|\cf|\le \binom{n-1}{k-1}$. 
Furthermore, when $n \ge 2k+1$, equality in the EKR theorem is achieved if and only if $\cf$ is a complete star.
Hilton and Milner~\cite{HM67} proved a strong stability version of the EKR theorem, showing that when $n\geq 2k+1$, every non-trivial intersecting family of $k$-subsets of $[n]$ has size at most $\binom{n-1}{k-1}-\binom{n-k-1}{k-1}+1$.

The EKR theorem has been studied extensively and many generalizations have been obtained; see the survey~\cite{DF83} for more details.
The degree version of the EKR theorem is one of the most extensively studied extensions in recent years.
One may view $\cf$ as a $k$-uniform hypergraph on the vertex set $[n]$, where the members of $\cf$ are its edges.
Given a family $\cf$ of $k$-subsets of $[n]$, the \emph{$S$-degree} of $\cf$ for $S \subseteq [n]$ is the number of sets in $\cf$ that contain $S$, denoted by $d_\cf(S)=|\{T \mid S \subseteq T \in \cf\}|$. 
For $1 \le d \le k$, the \emph{minimum $d$-degree} $\delta_d(\cf)$ is the smallest value of $d_\cf(S)$ among all $d$-subsets $S \subseteq [n]$.
Note that $\delta_1(\cf)$ is the minimum vertex degree of $\cf$.
Huang and Zhao~\cite{HZ17} proved that for $n> 2k$, every intersecting family $\cf\subseteq \binom{[n]}{k}$ has $\delta_1(\cf) \le \binom{n-2}{k-2}$, which implies the EKR theorem.
Later, Frankl and Tokushige~\cite{FT17} provided an elementary combinatorial proof of Huang and Zhao's result~\cite{HZ17} for the case $n\ge 3k$ based on the early work of Frankl~\cite{Fr87}.
The degree version of the Hilton--Milner theorem was proved by Frankl, Han, Huang, and Zhao~\cite{FHHZ18}.
For general $d$-degrees, Kupavskii~\cite{Ku19} proved a bound for sufficiently large $n$ by showing the following: If $1\le d<k$ and $n\geq 2k +3d/(1-d/k)$, then $\delta_d(\cf) \le \binom{n-d-1}{k-d-1}$ for every intersecting family $\cf \subseteq \binom{[n]}{k}$. 
Note that the lower bound on $n$ can be quadratic in $d$ and $k$ when $d$ is very close to $k$.
Huang and Zhang~\cite{HZ26} proved the following theorem, which improves the bound by Kupavskii.

\begin{theorem}[\cite{HZ26}]\label{thm:HZ}
Given positive integers $k$ and $d$ satisfying $k>d\ge 2$, 
let $\mathcal{F} \subseteq \binom{[n]}{k}$ be an intersecting family with $n \geq 2k+2d-3$.
    Then, 
    $$ \delta_d(\mathcal{F}) \leq \binom{n-d-1}{k-d-1}. $$
\end{theorem}

They also showed that unlike the case of the EKR theorem, $n \geq 2k$ is not sufficient to guarantee the same bound on $\delta_d(\cf)$ for $d=2$, and conjectured that $n \geq 2k+1$ is sufficient for all $d\ge 2$.
The bound in \Cref{thm:HZ} is asymptotically best possible when $d$ is close to $1$, but it is far from tight when $d$ is close to $k$.
Motivated by this, we focus on the case where $d$ is very close to $k$.
We first improve the bound on $n$ when $d=k-1$.
Our main theorem is the following.
\begin{theorem}\label{thm:main_k-1}
    Given a positive integer $k\ge 2$, let $\mathcal{F} \subseteq \binom{[n]}{k}$ be an intersecting family with $n \geq 2k + \left\lceil\frac{\sqrt{8k+1}-1}{2}\right\rceil+3$.
    Then, 
    $$ \delta_{k-1}(\mathcal{F}) \leq 1. $$
    Furthermore, equality holds if and only if $\cf$ is a complete star.
\end{theorem}

By extending our approach, we also improve the bound for $d=k-2$.

\begin{theorem}\label{thm:main_k-2}
    Given a positive integer $k\ge 3$, let $\mathcal{F} \subseteq \binom{[n]}{k}$ be an intersecting family with $n \geq 2k+7k^{2/3}+34k^{1/3}+162$.
    Then, 
    $$\delta_{k-2}(\mathcal{F}) \leq n-k+1. $$
    Furthermore, equality holds if and only if $\cf$ is a complete star.
\end{theorem}

The rest of the paper is organized as follows. In \Cref{sec:proof_main}, we prove \Cref{thm:main_k-1}, and in \Cref{sec:proofs_k-2}, we prove \Cref{thm:main_k-2}.

\section{Proof of Theorem~\ref{thm:main_k-1}}\label{sec:proof_main}

Our proof of \Cref{thm:main_k-1} proceeds in two steps. In the first step, we find a subfamily $\mathcal{H}$ of $\mathcal{F}$ such that $\mathcal{H}$ has at most one vertex cover of size $1$ and $|V(\mathcal{H})|$ is small. Note that if we view $\mathcal{H}$ as a subhypergraph of $\mathcal{F}$, then $V(\mathcal{H})$ is defined as the union of all edges in $\mathcal{H}$ unless we specify.
In the second step, we show that the restriction of $\cf$ to $V(\mathcal{H})$ is a complete star, and that it can be extended to a complete star on $V(\cf)$.
The following lemma addresses the first step.

\begin{lemma}\label{lem:k-1_vertex_cover}
    Given a positive integer $k\ge 2$, let $\mathcal{F} \subseteq \binom{[n]}{k}$ be an intersecting family with $n \geq k+\left\lceil\frac{\sqrt{8k+1}-1}{2}\right\rceil+2$ and $\delta_{k-1}(\cf)\ge 1$.
    Suppose $e\in\cf$ is an edge.
    Then there exists a subhypergraph $\mathcal H\subseteq \cf$ containing $e$ such that
    \[
    \left|\bigcap_{f\in\mathcal H}f\right|\le 1
    \quad\text{and}\quad
    |V(\mathcal H)|\le k+\left\lceil\frac{\sqrt{8k+1}-1}{2}\right\rceil+2.
    \]
\end{lemma}
\begin{proof}
    If $k=2$, then we can take $e$ and arbitrary edge intersecting $e$ to form $\mathcal{H}$, and the lemma holds. 
    Assume that $k\ge 3$.
    We construct a sequence of subhypergraphs $\mathcal{F}_0 \subseteq \mathcal{F}_1 \subseteq \cdots \subseteq \mathcal{F}$ inductively as follows.
    Let $\mathcal{F}_0$ be the subhypergraph of $\mathcal{F}$ with edge set $\{e\}$ and vertex set $e \cup \{v\}$ for some vertex $v \notin e$.
    For each $i \ge 0$, let
    \[
    U_i := \bigcap_{f \in \mathcal{F}_i} f, \quad u_i := |U_i|, \quad \text{and} \quad v_i := |V(\mathcal{F}_i)|.
    \]
    Note that $U_0 = e$, $u_0 = k$, and $v_0 = k+1$.

Assume that $\mathcal{F}_i$ has been constructed with $u_i \geq 2$.
    If $|V(\mathcal{F}_i) \setminus U_i| \geq k-1$, choose $W$ to be a $(k-1)$-subset of $V(\mathcal{F}_i) \setminus U_i$.
    If $|V(\mathcal{F}_i) \setminus U_i| < k-1$, choose $W$ to be a $(k-1)$-subset containing all vertices of $V(\mathcal{F}_i) \setminus U_i$ and some vertices of $U_i$.
    Since $\delta_{k-1}(\mathcal{F}) \geq 1$, there exists an edge $f_{i+1} \in \mathcal{F}$ such that $W \subseteq f_{i+1}$.
    Set
    \[
    \mathcal{F}_{i+1} := \mathcal{F}_i \cup \{f_{i+1}\}, \quad U_{i+1} := U_i \cap f_{i+1}.
    \]

    If $|V(\mathcal{F}_i) \setminus U_i| \geq k-1$, we have chosen $W \subseteq V(\mathcal{F}_i) \setminus U_i$; hence $u_{i+1} \leq 1$, and the construction terminates.
    If $|V(\mathcal{F}_i) \setminus U_i| < k-1$, then
    \[
    |U_i \setminus W| = u_i - (k-1 - |V(\mathcal{F}_i) \setminus U_i|) = v_i - k + 1.
    \]
    Since $f_{i+1}$ contains $W$, it contains exactly one vertex outside $W$, and thus at most one vertex of $U_i \setminus W$ belongs to $f_{i+1}$.
    Hence $u_{i+1} \leq u_i - (|U_i \setminus W| - 1) = u_i - (v_i - k)$.
    Letting $d_i := v_i - k$, we have $u_{i+1} \leq u_i - d_i$.

We observe that $d_0 = 1$, and each step adds at most one new vertex, so $d_{i+1} - d_i = v_{i+1} - v_i \in \{0, 1\}$.
    Let $t$ be the smallest index such that $u_t \leq 1$ or $|V(\mathcal{F}_t) \setminus U_t| \geq k-1$.
    Then, as $\sum_{i=0}^{t-1} (u_i - u_{i+1}) = u_0 - u_t \leq k$, we have $\sum_{i=0}^{t-1} d_i \leq k$.
    Let $D := d_{t-1}$.
    Since $d_0 = 1$ and the increments are at most $1$, the values $1, 2, \dots, D$ each appear at least once among $d_0, \dots, d_{t-1}$.
    Therefore,
    \[
    \frac{D(D+1)}{2} \leq \sum_{i=0}^{t-1} d_i \leq k,
    \]
    which yields $D \leq \left\lfloor \frac{\sqrt{8k+1} - 1}{2} \right\rfloor$.
    If $u_t \leq 1$, then we stop and set $\mathcal{H} = \mathcal{F}_t$.
    If $|V(\mathcal{F}_t) \setminus U_t| \geq k-1$, then by the previous argument, we have $u_{t+1} \leq 1$, so we can stop at $t+1$ and set $\mathcal{H} = \mathcal{F}_{t+1}$.
    In either case, we have
    \[
    |V(\mathcal{H})| \leq k + d_t + 1 \leq k + d_{t-1} + 2 \leq k + \left\lceil \frac{\sqrt{8k+1} - 1}{2} \right\rceil + 2.
    \]
    Moreover, $e \in \mathcal{H}$ and $\left| \bigcap_{f \in \mathcal{H}} f \right| \leq 1$, which completes the proof.
\end{proof}

Now we are ready to prove \Cref{thm:main_k-1}. For a family $\cf$ of $k$-subsets of $[n]$ and a subset $X \subseteq [n]$, we denote by $\cf[X]$ the family of all edges in $\cf$ that are contained in $X$.

\begin{proof}[Proof of \Cref{thm:main_k-1}]
    We choose an arbitrary edge $e \in \cf$ and apply \Cref{lem:k-1_vertex_cover} to obtain a subhypergraph $\mathcal{F}'\subseteq \cf$ containing $e$ such that $\mathcal{F}'$ has at most one vertex cover of size $1$ and $|V(\mathcal{F}')| \leq k+\left\lceil\frac{\sqrt{8k+1}-1}{2}\right\rceil+2$.
    Let $X = V(\mathcal{F}')$ and $X'$ be an arbitrary set containing $X$ with $|X'| = n-k > |X|$.
    We first prove that $\mathcal{F}[X']$ is a complete star.
    If $\mathcal{F}'$ does not have a vertex cover of size $1$, then any $(k-1)$-subset of $V(\mathcal{F}) \setminus X'$ has codegree $0$, which is a contradiction. 
    Thus, there is a unique vertex $v \in V(\mathcal{F}')$ such that all edges of $\mathcal{F}'$ contain $v$.
    Let $W\subseteq X' \setminus \{v\}$ be an arbitrary $(k-1)$-subset and let $f \in \cf$ be an edge containing $W$.
    If $f\neq W \cup \{v\}$, let $\cf'' := \cf' \cup \{f\}$.
    Then $\mathcal{F}''$ does not have a vertex cover of size $1$, so any $(k-1)$-subset of $V(\mathcal{F}) \setminus (V(\mathcal{F}'') \cup X')$ has codegree $0$, which is a contradiction.
    Therefore, $\mathcal{F}[X']$ is a complete star centered at $v$.

    We now choose an arbitrary $(k-1)$-subset $W\subseteq V(\mathcal{F}) \setminus X'$. 
    We note that, since $d_{\cf}(W) > 0$ and $\mathcal{F}$ is intersecting, $g:=W \cup \{v\}$ is an edge of $\mathcal{F}$.
    Then, by applying \Cref{lem:k-1_vertex_cover} to $g$ instead of $e$, 
    we obtain a subhypergraph $\mathcal{H}\subseteq \cf$ containing $g$ such that $\mathcal H$ has at most one vertex cover of size $1$ and $|V(\mathcal H)|\le k+\left\lceil \frac{\sqrt{8k+1}-1}{2}\right\rceil+2$.
    If there is a vertex $w\in V(\cf)\setminus(X'\cup V(\mathcal H))$, then we define a hypergraph $\mathcal{H}'$ with $V(\mathcal{H}'):= V(\mathcal{H})\cup \{w\}$ and $E(\mathcal{H}') := E(\mathcal{H})$.
    If such a vertex $w$ does not exist, we simply let $\mathcal{H}' = \mathcal{H}$.
    Then, by applying the same argument as in the first paragraph, we can show that there exists an $(n-k)$-subset $Y \supseteq V(\mathcal{H}')$ such that $\mathcal{F}[Y]$ is a complete star centered at $v'$ for some $v'\in V(g)$.
    Note that $X' \cup Y = V(\mathcal{F})$. 
    Suppose that $v\neq v'$.
    As $X'$ and $Y$ have size at least $k+1$ and $|X' \cup Y|\geq 2k+2$, we can choose $(k-1)$-subsets $W \subseteq X' \setminus \{v, v'\}$ and $W' \subseteq Y \setminus \{v, v'\}$ such that $W \cap W' = \emptyset$.
    Then $W \cup \{v\}$ and $W' \cup \{v'\}$ are two disjoint edges of $\mathcal{F}$, which is a contradiction.
    Therefore, $v = v'$.

    We finally claim that $\mathcal{F}$ is a complete star centered at $v$. 
    Let $\ell = \left\lceil\frac{\sqrt{8k+1}-1}{2}\right\rceil-1$. 
    Recall that $|X' \cap Y| = n-2k \geq  \ell+3$. 
    We arbitrarily partition $(X' \cap Y) \setminus \{v\}$ into two sets $Z_1$ and $Z_2$ such that each has size at least $\lceil \ell/2 \rceil+1$, and let $X_0 = (X' \setminus Y) \cup Z_1$ and $Y_0 = (Y \setminus X') \cup Z_2$.
    We note that $|X_0|, |Y_0| \geq k+\lceil \ell/2 \rceil+1$ and they are disjoint. 
    We proceed by mathematical induction on $t$ to prove the following: 
    For every $(k-1)$-subset $W \subseteq V(\mathcal{F}) \setminus \{v\}$, if $|W \cap X_0| \leq t \lceil \ell/2 \rceil$ or $|W \cap Y_0| \leq t \lceil \ell/2 \rceil$, then $W \cup \{v\}$ is an edge of $\mathcal{F}$.
    
    The base case $t=0$ holds because $\mathcal{F}[X_0 \cup \{v\}]$ and $\mathcal{F}[Y_0 \cup \{v\}]$ are complete stars centered at $v$.
    Suppose that the statement is true for $t-1$, and let $W \subseteq V(\mathcal{F}) \setminus \{v\}$ be a $(k-1)$-subset such that $|W \cap X_0| \leq t \lceil \ell/2 \rceil$.
    Suppose $W \cup \{v\}$ is not an edge. Then there exists a vertex $w \neq v$ such that $W \cup \{w\}$ is an edge.
    Then $|X_0 \setminus (W \cup \{w\})| \geq |X_0| - t\lceil \ell/2 \rceil -1 \geq k-(t-1)\lceil \ell/2 \rceil$.
    Therefore, there exists a $(k-1)$-subset $U \subseteq (X_0 \cup Y_0) \setminus (W \cup \{w\})$ such that $|U \cap X_0| \geq k-1-(t-1)\lceil \ell/2 \rceil$.
    Since $X_0$ and $Y_0$ are disjoint, $|U \cap Y_0| \leq (t-1)\lceil \ell/2 \rceil$, and $U$ is disjoint from $W \cup \{w\}$. 
    By the induction hypothesis, $U \cup \{v\}$ is an edge of $\mathcal{F}$.
    Then $W \cup \{w\}$ and $U \cup \{v\}$ are two disjoint edges of $\mathcal{F}$, which is a contradiction.
    Therefore, $W \cup \{v\}$ is an edge of $\mathcal{F}$.
    By symmetry, the same argument holds when $|W \cap Y_0| \leq t\lceil \ell/2 \rceil$.
    Therefore, by induction, $W \cup \{v\}$ is an edge of $\mathcal{F}$ for every $(k-1)$-subset $W \subseteq V(\mathcal{F}) \setminus \{v\}$, which implies that $\mathcal{F}$ contains a complete star centered at $v$.
    Finally, since $\mathcal{F}$ is intersecting and $n \geq 2k+1$, $\mathcal{F}$ must be exactly the complete star centered at $v$.
\end{proof}
    

\section{Proof of Theorem~\ref{thm:main_k-2}}\label{sec:proofs_k-2}

Let $\cf$ be a family of $k$-sets, and let $S$ be a set with $|S|<k$. Denote by $N_{\cf}(S)$ the family of $(k-|S|)$-sets $T$ such that $T\cup S \in \cf$. For $A\subseteq V(\cf)$, let $P_{\cf}(A)$ be the family of all size-two vertex covers of $\cf$ contained in $A$. Note that $P_{\cf}(A)$ can be viewed as a graph on the vertex set $A$.
Denote by $Q$ the disjoint union of an edge and a copy of $K_{1, 2}$.

\begin{lemma}\label{lem:edge_cherry_2}
    Let $k\geq 3$ and $n \geq k+2$. 
    Let $\cf \subseteq \binom{[n]}{k}$ be an intersecting family, and let $\cf'\subseteq \cf$ be a subfamily with disjoint sets $S, A \subseteq V(\cf')$ such that $|S|=k-2$. 
    Then the following statements hold.
    \begin{enumerate}[(1)]
        \item $N_{\cf'}(S)$ and $P_{\cf'}(A)$ are cross-intersecting.
        \item If $|N_{\cf}(S)|\geq n-k+1$ and $n \geq k+5$, then either $N_{\cf}(S)$ is a complete star on $V(\cf) \setminus S$ or there exists a family $\cf''\supseteq \cf'$ such that $|V(\cf'')|\leq |V(\cf')|+6$ and $P_{\cf''}(A)$ is a subgraph of a cherry. 
    \end{enumerate}
\end{lemma}
\begin{proof}
    (1): Since $S\cap A=\emptyset$, every member of $N_{\cf'}(S)$ must intersect every size-two vertex cover of $\cf'$ contained in $A$ because of the intersecting property. Thus, $N_{\cf'}(S)$ and $P_{\cf'}(A)$ are cross-intersecting.

    (2): Assume that $|N_{\cf}(S)|\geq n-k+1$.
    If $N_{\cf}(S)$ is a complete star on $V(\cf) \setminus S$, then we are done. Suppose then that $N_{\cf}(S)$ is not a complete star.
    Then $N_{\cf}(S)$ contains a matching of size two, and since there are $|N_{\cf}(S)| \geq n-k+1 \geq 6$ edges, $N_{\cf}(S)$ has a subgraph $F$ isomorphic to one of the following three configurations: a matching of size three, $Q$, or $K_4$.

    For each edge $xy\in E(F)$, let $f_{xy}:=S\cup\{x,y\}\in \cf$. 
    Add all such edges that are not already in $\cf'$ to $\cf'$, and denote the resulting family by $\cf''$.
    Since $F$ uses at most six vertices outside $S$, we have
    \[
    |V(\cf'')|\le |V(\cf')|+6.
    \]

    By part (1), $N_{\cf''}(S)$ and $P_{\cf''}(A)$ are cross-intersecting.
    Now we consider the chosen configuration $F$.
    If $F$ is a matching of size three, then no two-element set can intersect all three pairwise disjoint edges, so $P_{\cf''}(A)=\emptyset$.
    If $F=K_4$, then no two-element set can intersect all six edges of $K_4$, so again $P_{\cf''}(A)=\emptyset$.
    If $F=Q$ (an edge disjoint from a cherry), then every member of $P_{\cf''}(A)$ must intersect all three edges of $Q$; hence $P_{\cf''}(A)$ is contained in the two edges joining the center of the cherry to the two endpoints of the disjoint edge. Therefore, $P_{\cf''}(A)$ is a subgraph of a cherry.

    In every case, either $N_{\cf}(S)$ is a star, or there exists a family $\cf''\supseteq \cf'$ with $|V(\cf'')|\le |V(\cf')|+6$ such that $P_{\cf''}(A)$ is a subgraph of a cherry.
\end{proof}

Similar to the proof of \Cref{thm:main_k-1}, our main proof proceeds in two phases. In the first phase, we construct a subhypergraph $\cf'\subseteq \cf$ such that all size-two vertex covers of $\cf'$ have a common vertex. 
In the second phase, we use this property of $\cf'$ to show that $\cf$ is a complete star.
The following lemma is the key to the first phase.

\begin{lemma}\label{lem:k-2_vertex_cover}
    Let $\mathcal{F} \subseteq \binom{[n]}{k}$ be an intersecting family.
    Let $n \geq 2k+7k^{2/3}+34k^{1/3}+160$.
    If $\delta_{k-2}(\mathcal{F}) \geq n-k+1$, then for every edge $e\in \cf$, there exists a subhypergraph $\mathcal{F}'\subseteq \mathcal{F}$ containing $e$ such that $|V(\mathcal{F}')|\leq k+7k^{2/3}+34k^{1/3}+160$, and there exists a vertex $v\in e$ such that every size-two vertex cover of $\cf'$ contains $v$.
\end{lemma}
\begin{proof}
    Let $e \in \cf$ be an arbitrary edge, and let $x$ be a positive integer to be determined later. Note that we choose $x$ to be a function of $k$ such that $x \geq 10$ for all $k \geq 3$.
    Our construction consists of two phases.
    In Phase 1, our goal is to construct a subhypergraph $\cf''$ such that $\left|\bigcap_{f \in \cf''} f\right| \leq 1$ and $|V(\cf'')|\leq k+x + 2 \lceil k/x \rceil +4$. 
    To do so, we inductively construct a sequence of subhypergraphs $\cf_0 \subseteq \cf_1 \subseteq \dots \subseteq \cf$.
    Let $\cf_0$ be the subhypergraph with edge set $\{e\}$ and vertex set $e \cup X$, where $X$ is an arbitrary set of $x$ vertices disjoint from $e$. Let $U_0 = e$.
    Suppose that we have already constructed $\cf_i$ with common intersection $U_i = \bigcap_{f \in \cf_i} f$.
    We first assume that $|U_i| \geq 3$. 
    If $|U_i| \geq x+2$, let $V_i \subseteq U_i$ be an arbitrary subset of size $x+2$.
    Then $V(\cf_0) \setminus V_i$ is a $(k-2)$-subset. 
    By the minimum degree condition, there exists an edge $e_{i+1} \in \cf$ containing $V(\cf_0) \setminus V_i$.
    Let $\cf_{i+1} = \cf_i \cup \{e_{i+1}\}$.
    In this case, we obtain $x \leq |U_i \setminus U_{i+1}| \leq x+2$.
    If $3 \leq |U_i| \leq x+1$, then we take $V_i$ to be an arbitrary subset of size $x+2$ in $V(\cf_0)$ containing $U_i$. 
    Then we obtain $|U_{i+1}| \leq 2$.
    Note that in either case, $|V(\cf_{i+1}) \setminus V(\cf_{i})| \leq 2$.
    Thus, after $\ell \leq \lceil k/x \rceil +1$ steps, we obtain $|U_{\ell}| \leq 2$.
    
    If $|U_{\ell}| \leq 1$, then we are done.
    If $|U_{\ell}| = 2$, let $W$ be an arbitrary $(k-2)$-subset of $V(\cf_{\ell}) \setminus U_{\ell}$. Since $d_{\cf}(W) \ge n-k+1$, $N_{\cf}(W)$ contains a matching of size two or $N_{\cf}(W)$ is a star centered at some vertex $v \in V(\cf)$.
    In the former case, let $f_1, f_2 \in \mathcal{F}$ be two edges containing $W$ corresponding to the matching. Then $\cf_{\ell +1} := \cf_{\ell} \cup \{f_1, f_2\}$, and thus $U_{\ell +1} = \emptyset$.
    In the latter case, we select an edge $f \in \cf$ containing $W \cup \{v\}$ and a vertex not in $U_{\ell}$. Then, by setting $\cf_{\ell +1} := \cf_{\ell} \cup \{f\}$, we have $U_{\ell +1} \subseteq \{v\} \cap U_{\ell}$, which has size at most one.
    In both cases, we add at most four vertices to $\cf_{\ell}$, and thus $|V(\cf_{\ell +1})| \leq k+x + 2\lceil k/x \rceil + 4$.
    By adding additional isolated vertices if necessary, we also assume that $|V(\cf_{\ell +1})| \geq k+x$. 
    By setting $\cf'' := \cf_{\ell +1}$, we conclude the first phase.

    In the second phase, we construct the desired hypergraph $\cf'$ by adding edges to $\cf''$.
    If $|U_{\ell +1}| = 1$, let $v$ be its unique vertex. If $|U_{\ell +1}| = 0$, let $v$ be an arbitrary vertex in $e \subseteq V(\cf_{\ell +1})$.
    Note that any size-two vertex cover of $\cf_{\ell +1}$ that uses a vertex not in $V(\cf_{\ell +1})$ must intersect $U_{\ell +1}$, so it must contain $v$.
    We partition $V(\cf_{\ell+1}) \setminus \{v\}$ into $s$ parts $A_1,\dots, A_s$, each of size at at most $x/4$. Since $|V(\cf_{\ell+1})|\le k+x + 2\lceil k/x\rceil + 4$, we have
    \[
    s\le \left\lceil\frac{4(k+x+2\lceil k/x\rceil+3)}{x}\right\rceil.
    \]
    For each pair of indices $1 \leq i<i' \leq s$, we apply the following procedure. Suppose that we have already constructed $\cf_j$ for some $j\geq \ell +1$. 
    Let $A = A_i \cup A_{i'}$, and for any two distinct indices $r, r' \in [s] \setminus \{i, i'\}$, let $S_{r, r'}$ be a $(k-2)$-subset of $\bigcup_{t=1}^s A_t \setminus (A_i \cup A_{i'} \cup A_{r} \cup A_{r'})$. Note that such a set $S_{r, r'}$ exists because $|\bigcup_{t=1}^s A_t| \geq k+x-1$. 
    We apply \Cref{lem:edge_cherry_2} to $A$ and $S_{r, r'}$ and distinguish three cases depending on the structure of $N_{\cf}(S_{r, r'})$.

    \medskip
    
    \textbf{Case 1:} If there exist $r, r'$ such that $N_{\cf}(S_{r, r'})$ is not a star, then by \Cref{lem:edge_cherry_2}, there exists a supergraph $\cf_{j+1} \supseteq \cf_j$ such that $P_{\cf_{j+1}}(A)$ is a subgraph of a cherry. In this case, we define this family to be $\cf_{j+1}$ and proceed to a different choice of $i, i'$.

    \textbf{Case 2:} $N_{\cf}(S_{r, r'})$ is a star for every possible choice of $r$ and $r'$, and there exist two choices $S, S'$ such that $N_{\cf}(S)$ and $N_{\cf}(S')$ are stars with different centers $w$ and $w'$, respectively.
    We define $\cf_{j+1}$ by adding to $\cf_j$ all edges of the form $S \cup \{w\} \cup \{z\}$ for $z \in V(\cf_{\ell+1}) \setminus (S \cup \{w\})$, and all edges of the form $S' \cup \{w'\} \cup \{z\}$ for $z \in V(\cf_{\ell+1}) \setminus (S' \cup \{w'\})$. 
    Then $P_{\cf_{j+1}}(A)$ contains at most one element, namely $\{w, w'\}$. 
    Indeed, for any size-two subset $T \subseteq A$, if $T$ does not contain $w$ (or $w'$, respectively), then there is an edge of $\cf_{j+1}$ containing $S$ (or $S'$, respectively), $w$, and a vertex not in $T$. Thus, $T$ cannot be a size-two vertex cover of $\cf_{j+1}$.
    We also note that $|V(\cf_{j+1})| \leq |V(\cf_j)| + 2$.
    We then keep $\cf_{j+1}$ and proceed to a different choice of $i, i'$.

    \textbf{Case 3:} $N_{\cf}(S_{r, r'})$ is a star for every possible choice of $r$ and $r'$, and they are all centered at the same vertex $w$.
    In this case, we define $\cf_{j+1}$ by adding all edges contained in $S_{r, r'} \cup A \cup \{w\}$ to $\cf_j$ for every possible choice of $r$ and $r'$.
    We claim that all size-two vertex covers of $\cf_{j+1}$ contain $w$; i.e., $\cf_{j+1}$ is the desired subhypergraph.
    Let $T \subseteq V(\cf)$ be a size-two set which does not contain $w$.
    Then we can choose $r$ and $r'$ such that $S_{r, r'}$ is disjoint from $T$ by choosing $r$ and $r'$ be the indices such that $A_r$ and $A_{r'}$ intersects $T$ (if $r$ or $r'$ does not exist, we choose the remaining index arbitrarily).
    Consequently, since $V(\mathcal{F}_{j+1}) \setminus (S_{r, r'} \cup \{w\})$ has more than three vertices, there exists an edge of $\cf_{j+1}$ that contains $S_{r, r'}$ and is disjoint from $T$. 
    Therefore, $T$ cannot be a size-two vertex cover of $\cf_{j+1}$.
    Thus, all size-two vertex covers of $\cf_{j+1}$ contain $w$, so $\cf_{j+1}$ is the desired subhypergraph.
    In this case, we terminate the process and set $\cf' = \cf_{j+1}$.
    \medskip

Suppose that this process completes for every possible choice of $i$ and $i'$ without terminating.
    Let $\mathcal{H}$ be the resulting subhypergraph of $\mathcal{F}$.
    We note that all size-two vertex covers of $\mathcal{H}$ either contain $v$ or belong to $P_{\mathcal{H}}(A_i \cup A_{i'})$ for some $i$ and $i'$.
    In addition, by construction, $P_{\mathcal{H}}(A_i \cup A_{i'})$ is a subgraph of a cherry for every $i$ and $i'$.
    Let $A$ be the union of all vertices contained in $P_{\mathcal{H}}(A_i \cup A_{i'})$ over all $i, i'$. Note that $|A|\le 3s^2$.
    Let $S$ be a $(k-2)$-subset of $V(\mathcal{H}) \setminus (A \cup \{v\})$. (If $|V(\mathcal{H}) \setminus (A \cup \{v\})| < k-2$, we append arbitrary vertices from $V(\mathcal{F})$ to $V(\mathcal{H})$ to ensure such a set $S$ exists.)
    By applying \Cref{lem:edge_cherry_2} to $A$ and $S$, either $N_{\cf}(S)$ is a complete star, or there exists a supergraph $\mathcal{H}' \subseteq \mathcal{F}$ with $|V(\mathcal{H}')| \leq |V(\mathcal{H})|+6$ such that $P_{\mathcal{H}'}(A)$ is a subgraph of a cherry.
    Depending on the structure of $N_{\cf}(S)$ and $P_{\mathcal{H}'}(A)$, we distinguish three cases. In each case, we extend the current hypergraph to $\cf'$ to eliminate all size-two vertex covers that do not contain $v$.

    \medskip

    \textbf{Case 1:} $N_{\cf}(S)$ is a star with center $u \neq v$. 
    We add to $\mathcal{H}$ all edges of $\cf$ that contain $S \cup \{u\}$ and are contained in $V(\mathcal{H})$ to obtain a family $\mathcal{H}'$.
    If $T$ is a size-two vertex cover of $\mathcal{H}'$ and $T$ belongs to $P_{\mathcal{H}}(A_i \cup A_{i'})$ for some $i$ and $i'$, then since $S$ is disjoint from $T$, $T$ is a size-two vertex cover of $\mathcal{H}'$ if and only if it contains $u$. There are at most $2s$ such sets $T$.
    For a size-two set $T$ containing $v$, it is not a vertex cover of $\mathcal{H}'$ if it is disjoint from $S \cup \{u\}$.
    Therefore, $\mathcal{H}'$ has at most $2s+(k-1) = 2s+k-1$ size-two vertex covers. Then a $(k-2)$-subset $W \subseteq V(\mathcal{F}) \setminus V(\mathcal{H}')$ has degree $d_{\cf}(W) \leq 2s+k-1$. Since $n \geq 2k+2s+5$, we have $2s+k-1 < n-k+1$, which contradicts the assumption.

    \textbf{Case 2:} $N_{\cf}(S)$ is a star with center $v$.
    We add to $\mathcal{H}$ all edges of $\cf$ that contain $S \cup \{v\}$ and are contained in $V(\mathcal{H})$ to obtain a family $\mathcal{H}'$.
    Then for any size-two set $T \in P_{\mathcal{H}'}(A_i \cup A_{i'})$ for some $i$ and $i'$, since it is disjoint from $S \cup \{v\}$, it cannot be a vertex cover of $\mathcal{H}'$.
    Therefore, all size-two vertex covers of $\mathcal{H}'$ contain $v$. By setting $\cf'=\mathcal{H}'$, we obtain the desired family.

    \textbf{Case 3:} $P_{\mathcal{H}'}(A)$ is a subgraph of a cherry. 
    Let $A'$ be the set of vertices contained in any element of $P_{\mathcal{H}'}(A)$, and let $S'$ be a $(k-2)$-subset of $V(\mathcal{H}') \setminus (A' \cup \{v\})$.
    If there exists an edge $e \in \cf$ containing $S'$ but not $v$, we add it to $\mathcal{H}'$ to obtain a family $\mathcal{H}''$.
    Then any size-two vertex cover of $\mathcal{H}''$ either belongs to $P_{\mathcal{H}'}(A)$ or contains $v$ and intersects $e$.
    Therefore, there are at most $k+2$ size-two vertex covers of $\mathcal{H}''$ in $V(\mathcal{F})$. Thus, any $(k-2)$-subset $W \subseteq V(\mathcal{F}) \setminus V(\mathcal{H}'')$ has degree $d_{\cf}(W) \leq k+2 < n-k+1$, which contradicts the assumption.
    Therefore, $N_{\cf}(S')$ is a star centered at $v$. By adding to $\mathcal{H}'$ any edge of $\cf$ that contains $S' \cup \{v\}$ and is disjoint from $A$, we obtain a family $\cf'$ such that all size-two vertex covers contain $v$.

    \medskip

    Note that in all cases, $v \in e$ since $e \in \cf'$.
    Finally, we estimate the size of the vertex set $V(\cf')$. 
    We first recall that $|V(\cf'')|\leq k+x + 2\lceil k/x\rceil + 4$ and
    \[
    s\le \left\lceil\frac{4(k+x+2\lceil k/x\rceil+3)}{x}\right\rceil.
    \]
    During the process, we add at most $3s^2$ vertices to $\cf''$ to obtain $\mathcal{H}$, and at most $6$ vertices to $\mathcal{H}$ to obtain $\cf'$. Thus, we have
    \[
    |V(\cf')|\leq k+x + 2\lceil k/x\rceil + 4 + 3s^2 + 6.
    \]
    By taking $x=\lceil 5k^{2/3}\rceil$, we have
    \[
    s\le \left\lceil\frac{4(k+x+2\lceil k/x\rceil+3)}{x}\right\rceil
    \leq \frac{4(k+x+2\lceil k/x\rceil+3)}{x}+1
    =\frac{4k}{x}+5+\frac{8\lceil k/x\rceil+12}{x}
    \leq \frac{4k}{x}+7,
    \]
    where the last inequality follows from $4\lceil k/x\rceil+6\le x$, which holds for $x=\lceil5k^{2/3}\rceil$ and $k\ge 3$.
    Therefore,
    \[
    |V(\cf')|\leq k+x + 2\lceil k/x\rceil + 10 + 3\left(\frac{4k}{x}+7\right)^2
    \leq k+7k^{2/3}+34k^{1/3}+160.\qedhere
    \]

    We finally remark that $v \in e$. Indeed, if $v \notin e$, then since $e$ is an edge of $\cf'$, every size-two vertex cover of $\cf'$ must intersect $e$. Therefore, there are at most $k$ size-two vertex covers of $\cf'$, and thus a $(k-2)$-subset $W \subseteq V(\cf) \setminus V(\cf')$ has degree $d_{\cf}(W) \leq k < n-k+1$, which contradicts the assumption.
\end{proof}

Note that this lemma already shows that $\delta_{k-2}(\cf) \leq n-k+1$, since we can take a $(k-2)$-subset $W \subseteq V(\cf)\setminus V(\cf')$ and observe that $N_{\cf}(W)$ is a subset of the size-two vertex covers of $\cf'$.
We now show that if $\delta_{k-2}(\cf) = n-k+1$, then $\cf$ is a complete star.

\begin{proof}[Proof of \Cref{thm:main_k-2}]
    By the previous observation, we already have $\delta_{k-2}(\cf)\le n-k+1$.
    Thus, it remains to prove uniqueness when equality holds.
    Assume $\delta_{k-2}(\cf)=n-k+1$.

    Choose an arbitrary edge $e\in \cf$.  
    By \Cref{lem:k-2_vertex_cover}, there exists a subhypergraph $\cf'\subseteq \cf$ containing $e$ with
    \[
    |V(\cf')|\le k+7k^{2/3}+34k^{1/3}+160\le n-k-2,
    \]
    and a vertex $v\in e$ such that every size-two vertex cover of $\cf'$ contains $v$.
    Let $X:=V(\cf')$, and let $X^*$ be an arbitrary set of size $n-k$ containing $X$.

    We first show that $\cf[X^*]$ is a complete star centered at $v$.
    Let $S\subseteq X^*\setminus\{v\}$ be a $(k-2)$-subset.
    Suppose there is an edge $f\in \cf$ containing $S$ but not $v$.
    Write $f=S\cup\{a,b\}$.
    Since $|V(\cf)\setminus X^*| = k$, we can choose a $(k-2)$-subset $W\subseteq (V(\cf)\setminus X^*)\setminus\{a,b\}$.
    Then $W$ is disjoint from every edge of $\cf'$ and also disjoint from $f$.
    Hence, for every $T\in N_{\cf}(W)$, the pair $T$ is a size-two vertex cover of $\cf'$, so $v\in T$.
    Also, since $W\cup T$ must intersect $f$ and $v\notin f$, the second vertex of $T$ must lie in $S\cup\{a,b\}$.
    Therefore, we have $d_{\cf}(W)=|N_{\cf}(W)|\le |S|+2=k$,
    which contradicts $d_{\cf}(W)\ge \delta_{k-2}(\cf)=n-k+1>k$.
    Thus, every edge containing $S$ contains $v$. 
    Together with the minimum degree condition, this implies that $N_{\cf}(S)$ is a star centered at $v$ for every $(k-2)$-subset $S\subseteq X^*\setminus\{v\}$, so $\cf[X^*]$ is a complete star centered at $v$.

    Let $B := V(\cf) \setminus X^*$, which is a set of exactly $k$ vertices.
    Choose an arbitrary $(k-2)$-subset $W \subseteq B$.
    Since $W \cap X^* = \emptyset$, every pair in $N_{\cf}(W)$ is a size-two vertex cover of $\cf[X^*]$, and thus contains $v$.
    Therefore, every edge containing $W$ contains $v$.
    Pick an edge $g = W \cup \{v, u\} \in \cf$.
    By applying the same construction to $g$, and again extending up to size $n-k$, we obtain an $(n-k)$-subset $Y \subseteq V(\cf)$ such that $\cf[Y]$ is a complete star centered at some $v'$, and we may choose the extension such that $X^* \cup Y = V(\cf)$.
    Since every edge containing $W$ contains $v$, the star on $Y$ forces $v' = v$.

    Let
    \[
    R := (X^* \cap Y) \setminus \{v\}, \qquad A := [n] \setminus Y, \qquad B := [n] \setminus X^*.
    \]
    Then $|A| = |B| = k$ and $|R| = n-2k-1$.
    Let $\ell =  \lceil k^{2/3} \rceil - 2$. We arbitrarily partition $R$ into two sets $Z_1$ and $Z_2$ such that $|Z_1|, |Z_2| \geq \lceil k^{2/3} \rceil$. We also define $X_0 := A \cup Z_1$ and $Y_0 := B \cup Z_2$.
    Then we have $|X_0|, |Y_0| \geq k+\ell+2$, $X_0 \cap Y_0 = \emptyset$, and both $\cf[X_0 \cup \{v\}]$ and $\cf[Y_0 \cup \{v\}]$ are complete stars centered at $v$.
    We claim that for every integer $t\ge 0$ and every $(k-2)$-subset $W\subseteq V(\cf)\setminus\{v\}$, if $|W\cap X_0|\le t\ell$ or $|W\cap Y_0|\le t\ell$, then $N_{\cf}(W)$ is a complete star centered at $v$.

    For the base case $t=0$, we have $|W\cap X_0|=0$ or $|W\cap Y_0|=0$. In the former case, $W$ is a subset of $Y_0$ and suppose that $N_{\cf}(W)$ is not a star centered at $v$. 
    Then there exists an edge $f\in \cf$ containing $W$ but not $v$, and we can write $f=W\cup\{a,b\}$.
    Then we can choose an edge $e$ of $\cf$ in $X_0 \cup \{v\}$ that is disjoint from $f$ as $|X_0 \cup \{v\}| \geq k+1$ and $\cf[X_0 \cup \{v\}]$ is a complete star centered at $v$.
    Then $e$ and $f$ are disjoint, which contradicts the intersecting property of $\cf$.
    Suppose that the claim holds for $t-1$, and let $W\subseteq V(\cf)\setminus\{v\}$ be a $(k-2)$-subset with $|W\cap X_0|\le t\ell$ (the other case is symmetric).
    Suppose that $N_{\cf}(W)$ is not a star centered at $v$.
    Then there exists an edge $f\in \cf$ containing $W$ but not $v$, and we can write $f=W\cup\{a,b\}$.
    Since $|W\cap X_0|\le t\ell$, we have $|X_0 \cap f| \leq t \ell + 2$, and thus $|X_0 \setminus f| \geq k-(t-1)\ell$. 
    Therefore, we can choose a $(k-2)$-subset $W' \subseteq (X_0 \cup Y_0) \setminus f$ such that $|W' \cap X_0| \geq k-2-(t-1)\ell$.
    Since $X_0$ and $Y_0$ are disjoint, we have $|W' \cap Y_0| \leq (t-1)\ell$.
    By the induction hypothesis, $N_{\cf}(W')$ is a star centered at $v$.
    Thus, there exists an edge of $\cf$ of the form $W' \cup \{v\} \cup \{z\}$ for some $z \notin f$.
    Then this edge $W' \cup \{v\} \cup \{z\}$ is disjoint from $f$, which contradicts the intersecting property of $\cf$.
    Therefore, $N_{\cf}(W)$ is a star centered at $v$.
    By mathematical induction, the statement holds for all $t$. Therefore, for every $(k-2)$-subset $W \subseteq V(\cf) \setminus \{v\}$, $N_{\cf}(W)$ is a star centered at $v$, which completes the proof. 
\end{proof} 

\section*{Acknowledgements}
This work was initiated when the last two authors visited BUPT (Hainan) in December 2024 and they would like to express their gratitude to the host institute for the nice working environment.

\end{document}